\documentclass[11pt,leqno]{amsart}
\usepackage[margin=1in]{geometry}
\usepackage[T1]{fontenc}
\usepackage{lmodern}
\usepackage{amsmath,amssymb,amsthm,mathtools}
\usepackage{microtype}
\usepackage{enumitem}
\usepackage{aliascnt}
\usepackage[hidelinks]{hyperref}
\usepackage[capitalize,noabbrev]{cleveref}

\numberwithin{equation}{section}
\newtheorem{theorem}{Theorem}[section]
\newaliascnt{lemma}{theorem}
\newtheorem{lemma}[lemma]{Lemma}
\aliascntresetthe{lemma}
\newaliascnt{proposition}{theorem}
\newtheorem{proposition}[proposition]{Proposition}
\aliascntresetthe{proposition}
\newaliascnt{corollary}{theorem}

\aliascntresetthe{corollary}
\theoremstyle{definition}
\newaliascnt{example}{theorem}

\aliascntresetthe{example}
\newaliascnt{definition}{theorem}

\aliascntresetthe{definition}
\theoremstyle{remark}
\newaliascnt{remark}{theorem}
\newtheorem{remark}[remark]{Remark}
\aliascntresetthe{remark}
\crefname{lemma}{Lemma}{Lemmas}
\crefname{proposition}{Proposition}{Propositions}
\crefname{corollary}{Corollary}{Corollaries}
\crefname{definition}{Definition}{Definitions}
\crefname{example}{Example}{Examples}
\crefname{remark}{Remark}{Remarks}

\newcommand{\R}{\mathbb{R}}
\newcommand{\Leb}{\mathcal{L}}
\newcommand{\lip}[2]{[#1]_{C^{0,1}(#2)}}

\newcommand{\dd}{\,\mathrm{d}}

\newcommand{\cA}{\mathcal A}
\setlist[enumerate]{label=\textnormal{(\roman*)},leftmargin=*,itemsep=3pt,topsep=5pt}
\hypersetup{
 pdftitle={A sharp reverse isoperimetric inequality for the planar Monge--Amp\`ere eigenvalue},
 pdfsubject={A sharp lower bound for the Monge--Ampere eigenvalue at fixed area},
 pdfauthor={},
 pdfkeywords={Monge-Ampere eigenvalue, reverse isoperimetric inequality, triangle, partial Legendre transform}
}

\title[Triangles minimize the Monge--Amp\`ere eigenvalues]
{Triangles minimize the Monge--Amp\`ere eigenvalue among planar bounded convex domains of a given area}
\author{Chong Gu}
\address{Department of Mathematics, Indiana University, Bloomington, IN 47405, USA}
\email{chongu@iu.edu}
\thanks{The author was supported in part by the National Science Foundation under grant DMS-2452320.}

\begin{document}

\begin{abstract}
We prove that triangles minimize the Monge--Amp\`ere eigenvalue among bounded planar convex domains of a given area, thus resolving the two-dimensional case of Le's simplex conjecture for the Monge--Amp\`ere eigenvalue (The eigenvalue problem for the Monge--Amp\`ere operator on general bounded convex domains, {\it Ann. Sc. Norm. Super. Pisa Cl. Sci.} (5) {\bf 18} (2018)). The proof uses a neighbor-decreasing property of the Monge--Amp\`ere eigenvalues of convex polygons along vertex elimination parallel chord movements. 
\end{abstract}
\maketitle

\section{Introduction and statement of the main result}\label{sec:framework}
In this paper, we prove that triangles minimize the Monge--Amp\`ere eigenvalue among bounded planar convex domains of a given area, thus resolving the two-dimensional case of Le's simplex conjecture for the Monge--Amp\`ere eigenvalue \cite[Conjecture~1.5(i)]{Le2018}. 

\medskip

The Monge--Amp\`ere eigenvalue problem on a bounded convex domain $\Omega \subset \R^n (n\geq 2)$
\begin{equation}
\label{MAp}
\det D^2 u = \lambda|u|^n \quad \text{in } \Omega, \qquad u = 0 \quad \text{on } \partial \Omega
\end{equation}
was first studied by 
Lions \cite{Lions1985} for the case of smooth, uniformly convex domains. 
He showed that there is a unique positive constant $\lambda=\lambda(\Omega)$ for which \eqref{MAp}
has a nonzero convex solution $u \in C^\infty(\Omega)\cap C^{1,1}(\overline{\Omega})$. The Monge--Amp\`ere eigenfunction $u$ is unique up to multiplication by positive constants. The global smoothness $u\in C^{\infty}(\overline{\Omega})$ was obtained by Hong--Huang--Wang \cite{HHW} in two dimensions and by Le--Savin \cite{LS} in all dimensions. 
Tso \cite{Tso1990} discovered a variational characterization of $\lambda(\Omega)$ via the Monge--Amp\`ere Rayleigh quotient.

\medskip
In \cite{Le2018},  Le extended the results of Lions and Tso to all general bounded convex domains $\Omega \subset \R^n$. He showed that 
there is a unique positive constant $\lambda=\lambda[\Omega]$ (called the Monge--Amp\`ere eigenvalue of $\Omega$) for which \eqref{MAp}
has a unique (up to multiplication by positive constants) nonzero convex solution $u \in C^\infty(\Omega)\cap C(\overline{\Omega})$. The function $u$ is called a  Monge--Amp\`ere eigenfunction of $\Omega$. As in \cite{Le2018}, the notation $\lambda[\Omega]$ for the Monge--Amp\`ere eigenvalue is used  to indicate that $\partial \Omega$ can possibly be nonsmooth and contain flat parts. Le also established the global Lipschitz regularity of the Monge--Amp\`ere eigenfunctions on general bounded convex domains \cite{Le2026}. Using this result, he showed that $\lambda[\Omega]$ can be characterized as the infimum of the Monge--Amp\`ere Rayleigh quotient
\[
\lambda[\Omega] = \inf \left\{ \frac{\int_\Omega |u| \det D^2 u \; dx}{\int_\Omega |u|^{n+1} \; dx}: u \in C^{0,1 }(\overline\Omega)\cap C^2(\Omega)\setminus \{0\}, u \text{ convex on } \Omega, u = 0 \text{ on } \partial \Omega\right\}.
\]
See also \cite[Chapter 11]{Le2024} for a self-contained exposition of the 
Monge--Amp\`ere eigenvalue problem.

\medskip

Salani~\cite{Salani2005} proved a Brunn--Minkowski inequality for the Monge--Amp\`ere eigenvalue on smooth uniformly convex domains.
Le~\cite[Theorem~1.3]{Le2018} extended this inequality to general bounded convex domains.
He later characterized the equality case~\cite[Theorem~1.3]{LeBM2026}.

\medskip
Brandolini--Nitsch--Trombetti~\cite{BNT2009} proved that ellipsoids maximize the Monge--Amp\`ere eigenvalue among smooth uniformly convex domains of a fixed volume.
Le~\cite[Theorem~1.4]{Le2018} extended this result to general bounded convex domains.

\medskip
Our problem here is to minimize $\lambda[\Omega]$ among bounded convex domains of a fixed volume.
Le~\cite[Theorem~1.4]{Le2018} proved that the minimum is attained.
He conjectured that simplices are minimizers~\cite[Conjecture~1.5(i)]{Le2018}.
We prove that triangles attain this minimum in the planar case, thus resolving Le's simplex conjecture for the Monge--Amp\`ere eigenvalue in the two-dimensional case.
Following \cite[Section~1.2]{Le2018}, we call the resulting lower bound a {\it reverse isoperimetric inequality for the Monge--Amp\`ere eigenvalue}.

\medskip
A triangle in this paper always means the interior of the convex hull of three noncollinear points in the plane $\R^2$. Our main result states as follows.

\begin{theorem}[Triangles minimize the Monge--Amp\`ere eigenvalue among bounded planar convex domains of a given area]\label{thm:main}
Let $\Omega\subset\R^2$ be a bounded convex domain with nonempty interior, and let $T\subset\R^2$ be a triangle with area $|T|=|\Omega|$. Then, their Monge--Amp\`ere eigenvalues satisfy
\begin{equation}
 \lambda[\Omega]\geq\lambda[T].
 \label{eq:main}
\end{equation}
\end{theorem}

Triangles of the same area have the same  Monge--Amp\`ere eigenvalue by the affine invariance; see Lemma \ref{lem:affine}. The theorem proves that they all attain the minimum; it does not classify every minimizing domain. We conjecture that triangles are the only minimizers.

\medskip

We would like to mention some closely related results, not to be exhaustive. 
Ball's reverse isoperimetric inequality~\cite{Ball1991} is concerned with the surface area of convex bodies in $\R^n$. He showed that
every convex body has an affine image with surface area at most that of a regular simplex of the same volume; for centrally symmetric bodies, the comparison is with a cube of the same volume.
Ball's proof uses John's characterization of inscribed ellipsoids of maximal volume~\cite{John1948} and the Brascamp--Lieb inequality \cite{BL}.

\medskip
Mahler's conjecture \cite{M1, M2} predicts that simplices minimize the product of the volumes of a convex body and its polar.
The polar is taken about the Santal\'o point, which minimizes the polar volume over all interior points. The two-dimensional case was solved by Mahler himself \cite{M1}.
For centrally symmetric bodies, cubes are conjectured minimizers.
Iriyeh and Shibata~\cite{IS2020} proved the three-dimensional symmetric case.
Recent work of Chen--Li--Xi--Xu~\cite{CLXX2026} presents a proof of the general three-dimensional case.
They also determine the equality case.

\medskip
Bucur and Fragal\`a~\cite{BFS2016} studied the corresponding product of the first Dirichlet Laplacian eigenvalues of a convex body and its polar.
They proved that balls minimize this product among origin-symmetric convex bodies.
For planar bodies symmetric about both coordinate axes, they also obtained a comparison with the square.
Each such body has an invertible diagonal image whose eigenvalue product is no larger than that of a square.

\medskip
We say a few words on the proof of Theorem \ref{thm:main}. A bounded planar convex domain can be approximated by convex polygons in the Hausdorff distance. By the continuity of the Monge--Amp\`ere eigenvalue $\lambda[\Omega]$, it suffices to prove Theorem \ref{thm:main} for $\Omega$ being a polygon. For a polygon with at least four vertices, we move a vertex parallel to the segment joining its neighbors and keep other vertices fixed; the area is preserved and the polygon remains convex. This translation is similar to the parallel chord movement investigated by Campi--Gronchi \cite{CG} and has its root in the shadow system introduced by Rogers-Shephard \cite{RS}.  In our main neighbor-decreasing result, we will show that this vertex can be moved to obtain a convex polygon with fewer vertices and no larger Monge--Amp\`ere eigenvalue. Repeating this process gives a triangle that minimizes the Monge--Amp\`ere eigenvalue.

\medskip
More technical details of the proof are as follows.
\begin{enumerate} 
\item We first express the Monge--Ampe\`re energy $I[u;\Omega]= \int_\Omega |u|\, d\mu_u$ of a convex function $u\in C(\overline{\Omega})$ vanishing on the boundary $\partial\Omega$ via the integral $\int_{\R^n} (u^\ast-h_\Omega)\, dx$ of  its Legendre transform and the support function $h_\Omega$ of $\overline{\Omega}\subset\R^n$. See Lemma \ref{lem:full-dual}.
\item When $P\subset\R^2$ is a polygon in the plane,  $I[u;P]$ can be expressed via $\int_\R\int_l^r [F_x^2(x, p)-H_x^2(x, p)] \, dx\, dp$, which is an integral on $\R$ of a difference $\int_l^r [F_x^2(x, p)-H_x^2(x, p)] \, dx$ of the one-dimensional Dirichlet energies of the partial Legendre transform $F(x, p)$ of $u$ in $y$ and the support function $H(x, p)$ of $P$ in the $y$-direction. See Proposition \ref{prop:envelope-energy}.
\item Let $P\subset\R^2$ be a polygon with at least $m\geq 4$ vertices. We move a vertex of $P$ parallel to its nearest diagonal by an area-preserving transformation using a piecewise affine function $\eta$ via $(x, y)\mapsto (x, y+ t\eta(x))$ where $t_0\leq t\leq t_1$ and $0\in (t_0, t_1)$,  and obtain convex polygons $P_t$ with $P_{t_0}$ and $P_{t_1}$ have at most $(m-1)$ vertices. We show that either $P_{t_0}$ or $P_{t_1}$ has Monge--Amp\`ere eigenvalue not greater than that of $P$. To do this, let $u$ be a nonzero Monge--Amp\`ere eigenfunction of $P$. 
 From the variational characterization of $\lambda[P_t]$, we need to choose a suitable convex function $U_t$ vanishing on $\partial P_t$ and $t\in (t_0, t_1)$ such that $I[U_t; P_t]/\|U_t\|_{L^3(P_t)}^3\leq \lambda[P]$. A possible candidate for $U_t$ is  $w_t(x, y)=u(x, y-t\eta(x))$, but this function maybe nonconvex. However, choosing
$U_t$ as a {\it convex envelope} of $w_t$ on $P_t$ will work. In fact, we can show that $|U_t\|_{L^3(P_t)}^3\geq |u\|_{L^3(P)}^3$ and
$I[U_t; P_t] \leq I[u; P] + tb$ where $b$ is a fixed constant, independent of $t$. The linear change comes from using (ii), and the fact that the partial Legendre transform $F_t(x, p)$ of $w_t$ in $y$ and the support function $H_t(x, p)$ of $P_t$ in the $y$-direction change linearly in $t$. See Theorem \ref{thm:endpoint-comparison}.

\end{enumerate}

Throughout the paper, a convex domain in $\R^n$ is a bounded open convex set with nonempty interior. Polygons are open, and their vertices are the extreme points of their closures. We write $|E|=\Leb^n(E)$ for the Lebesgue measure of a measurable set $E\subset\R^n$ and $|z|$ for the Euclidean norm of $z\in\R^n$. 

\medskip
For a convex function $u$ on $\Omega$, its subdifferential and Monge--Amp\`ere measure are
\[
 \partial u(z)=\{\xi\in\R^n:u(w)\geq u(z)+\xi\cdot(w-z)
                  \text{ for all }w\in\Omega\},
 \qquad \mu_u(E)=|\partial u(E)|,
\]
where $E\subset\Omega$ is Borel and $\partial u(E)=\bigcup_{z\in E}\partial u(z)$. If $u\in C^2(\Omega)$, then $\mu_u=(\det D^2u)\dd \Leb^n$. See \cite[Chapter 3]{Le2024}.

\medskip
Let $\Omega\subset\R^n$ be a bounded convex domain. Set
\begin{equation}
 \cA(\Omega)=\{u\in C(\overline{\Omega})\setminus\{0\}:u\text{ is convex},\ u=0\quad\text{on }\partial\Omega\},
 \label{eq:admissible-energy}
\end{equation}
and define the Monge--Amp\`ere energy of $u\in \cA(\Omega)$ as 
\[
I[u;\Omega]=\int_\Omega |u|\dd\mu_u.
\]
The energy may be infinite.  By \cite[Theorem 1.1]{Le2018}, the Monge--Amp\`ere eigenvalue of $\Omega$ has the following variational characterization
\begin{equation}
 \lambda[\Omega]=\inf_{u\in\cA(\Omega)}
 \frac{I[u;\Omega]}{\|u\|_{L^{n+1}(\Omega)}^{n+1}},
 \label{eq:rayleigh}
\end{equation}
and every Monge--Amp\`ere eigenfunction attains the infimum. 

\medskip
The global Lipschitz regularity of the Monge--Amp\`ere eigenfunctions on general bounded convex domains \cite[Theorem 1.1]{Le2026} allows us to choose test functions in the variational characterization of the Monge--Amp\`ere eigenvalue in the class of Lipschitz functions. This substantially facilitates our analysis.

\medskip For $K\subset\R^n$ and a Lipschitz function $f\in C^{0, 1}(K)$, its Lipschitz constant is defined by
\[ [f]_{C^{0, 1}(K)}:= \sup_{x\neq y\in K} \frac{|f(x)-f(y)|}{|x-y|}.\]

We also recall the affine invariance of $\lambda[\Omega]$; see \cite[Proposition 5.8]{Le2018}.
\begin{lemma}[Affine invariance]\label{lem:affine}
Let $\Omega\subset\R^n$ be a bounded convex domain, $A$ be an $n\times n$ invertible matrix, and $b\in\R^n$. Then
\begin{equation}
 \lambda[A\Omega+b]=|\det A|^{-2}\lambda[\Omega].
 \label{eq:affine}
\end{equation}
Thus simplices of the same volume have the same Monge--Amp\`ere eigenvalue. In dimension two, $\lambda[a\Omega]=a^{-4}\lambda[\Omega]$ for $a>0$.
\end{lemma}

In particular, $|\Omega|^2\lambda[\Omega]$ is affine invariant. The  comparison in \cref{thm:main} can therefore be written using one fixed triangle of area one.

\medskip
\noindent \textbf{Organization of the paper.} The rest of the paper is organized as follows. In Section \ref{sec:energy}, we express the Monge--Amp\`ere energy in terms of the Legendre transforms and obtain an energy estimate of convex envelopes by the Legendre
transforms and partial Legendre transforms. In Section \ref{sec:chords}, we apply this energy estimate to reduce polygons to triangles and prove Theorem \ref{thm:main}.

\medskip
\noindent\textbf{Acknowledgments.} The author would like to thank his advisor, Professor Nam Q. Le, for his patience and guidance throughout the course of this work.

\medskip
\noindent\textbf{AI assistance.} The main results of this paper were obtained through a series of chats with ChatGPT-6 Astra. The key strategies were obtained by ChatGPT. The author reworked and rewrote the article entirely. We take full responsibility for its correctness and content.

\section{Monge--Amp\`ere energy, Legendre transform, and convex envelope}\label{sec:energy}

In this section, we estimate in Proposition \ref{prop:envelope-energy} the Monge--Amp\`ere energy of the convex envelope of a globally Lipschitz function by the partial Legendre transform of the original function. When the original function is convex, the envelope equals that function and the estimate becomes an identity. The estimate in this section applies to any nonpositive Lipschitz function with zero boundary values on a polygon.

\subsection{The Legendre transform and the convex envelope}

Let $\Omega\subset\R^n$ be a convex domain. For $w\in C(\overline{\Omega})$, define the Legendre transform of $w$ and the support function of $\overline{\Omega}$ by
\begin{equation}
 w^*(\xi)=\max_{z\in \overline{\Omega}}\{\xi\cdot z-w(z)\},
 \qquad h_\Omega(\xi)=\max_{z\in \overline{\Omega}}\xi\cdot z.
 \label{eq:full-transform}
\end{equation}
Thus $w^*$ is the Legendre transform on $\R^n$ of $w$ extended by $+\infty$ outside $\overline{\Omega}$, and $h_\Omega$ is the Legendre transform of the zero function on $\overline{\Omega}$.

\medskip
The energy identity below is from \cite[Lemma 2.2]{Zhou2026}. That result holds for continuous convex functions with zero boundary values, with equality in $[0,+\infty]$. We state and prove the Lipschitz case used here.

\begin{lemma}\label{lem:full-dual}
Let $\Omega\subset\R^n$ be a bounded convex domain.
If $u\in C^{0,1}(\overline{\Omega})$ is convex and $u=0$ on $\partial\Omega$, then
\begin{equation}
 I[u;\Omega]=(n+1)\int_{\R^n}(u^*-h_\Omega)\dd\xi<\infty.
 \label{eq:full-dual}
\end{equation}
Moreover, $u^*-h_\Omega\geq0$ and
\begin{equation}
 u^*(\xi)=h_\Omega(\xi)\qquad\text{if }|\xi|>\lip{u}{\overline{\Omega}}.
 \label{eq:full-tail}
\end{equation}
\end{lemma}

\begin{proof}
Set $L=\lip{u}{\overline{\Omega}}$ and $\Psi=u^*-h_\Omega$.
If $z\in\Omega$ and $\xi\in\partial u(z)$, then
\[
 t\,\xi\cdot e\leq u(z+te)-u(z)\leq Lt
\]
for every unit vector $e$ and all sufficiently small $t>0$.
Taking $e=\xi/|\xi|$ when $\xi\ne0$ gives
\[
 \partial u(\Omega)\subset\{\xi\in\R^n:|\xi|\leq L\}.
\]
Thus \[\mu_u(\Omega)<\infty \quad\text{and} \quad I[u;\Omega]<\infty.\]
Since $u\leq0$, we have $\Psi\geq0$.

\medskip
Fix $\xi$ with $|\xi|>L$.
Suppose a maximizer $z$ defining $u^*(\xi)$ lies in $\Omega$.
Then
\[
 \xi\cdot z-u(z)\geq\xi\cdot w-u(w)
 \qquad\text{for all }w\in\Omega.
\]
Thus $\xi\in\partial u(z)$, which contradicts the subdifferential bound above.
Every maximizer therefore lies on $\partial\Omega$.
The linear function $z\mapsto\xi\cdot z$ also attains its maximum on $\partial\Omega$.
Since $u=0$ on $\partial\Omega$,
\[
 u^*(\xi)
 =\max_{z\in\partial\Omega}\xi\cdot z
 =\max_{z\in\overline{\Omega}}\xi\cdot z
 =h_\Omega(\xi).
\]
This proves \eqref{eq:full-tail}.
Both $u^*$ and $h_\Omega$ are Lipschitz because $\Omega$ is bounded.
Hence $\Psi$ is Lipschitz and compactly supported.

Since $u^*$ is locally Lipschitz in $\partial u(\Omega)$, by the Rademacher theorem, the gradient $Du^*(p)$ exists for $p\in \partial u(\Omega)$  almost everywhere in the sense of Lebesgue measure.   By \cite[Lemma~2.3]{LeBM2026}, 
\[
 I[u;\Omega]
 =\int_{\partial u(\Omega)}-u(D u^*(\xi))\dd\xi.
\]
At every differentiability point of $u^*$, the maximizer defining $u^*(\xi)$ is unique and equals $D u^*(\xi)\in\overline{\Omega}$.
If $\xi\notin\partial u(\Omega)$, this maximizer lies on $\partial\Omega$.
The zero boundary values therefore allow us to extend the integral to $\R^n$.
By the definition of $u^*$, we obtain
\[
 I[u;\Omega]
 =\int_{\R^n}\bigl(u^*(\xi)-\xi\cdot D u^*(\xi)\bigr)\dd\xi.
\]

The support function is positively homogeneous of degree one, so
$h_\Omega(\xi)=\xi\cdot D h_\Omega(\xi)$ almost everywhere.
It follows that
\[
 I[u;\Omega]
 =\int_{\R^n}\bigl(\Psi-\xi\cdot D\Psi\bigr)\dd\xi.
\]
Since $\Psi$ is Lipschitz and compactly supported, integration by parts gives
\[
 \int_{\R^n}\xi\cdot D\Psi\dd\xi
 =-n\int_{\R^n}\Psi\dd\xi.
\]
This proves \eqref{eq:full-dual}.
\end{proof}

\medskip
The following lemma allows us to apply the preceding energy identity to convex envelopes. Notice that the argument of the proof does not require the strict convexity of the domain.
\begin{lemma}[Convex envelope with zero boundary values]\label{lem:convex-envelope}
Let $\Omega\subset\R^n$ be a bounded convex domain.
Let $w\in C^{0,1}(\overline{\Omega})$ satisfy $w\leq0$ on $\overline{\Omega}$ and $w=0$ on $\partial\Omega$. Define the convex envelope of $w$ in $\overline\Omega$ by
\begin{equation}
 U(z)
 :=\sup\{\ell(z):\ell\text{ is affine on }\R^n,\ \ell\leq w\text{ on }\overline{\Omega}\}.
 \label{eq:convex-envelope}
\end{equation}
Then $U \in C^{0,1}(\overline{\Omega})$ is convex, and vanishes on $\partial\Omega$. Moreover,
\begin{equation}
 U\leq w\leq0,\qquad \lip{U}{\overline{\Omega}}\leq\lip{w}{\overline{\Omega}},\qquad U^*=w^*.
 \label{eq:convex-envelope-properties}
\end{equation}
If $w\not\equiv0$, then $U\in\cA(\Omega)$.
\end{lemma}

\begin{proof}
Set $L=\lip{w}{\overline{\Omega}}$ and $d(z)=\operatorname{dist}(z,\partial\Omega)$. Since $w = 0$ on $\partial \Omega$, we have $w\geq-Ld$. Since $d(z)$ is concave on bounded convex domains, $-Ld$ is convex. Hence, by the definition of $U$,
\begin{equation}
 -Ld\leq U\leq w\leq0\quad\text{on }\overline{\Omega}.
 \label{eq:distance-barrier}
\end{equation}
Thus $U$ is finite and convex, and extends continuously by zero to $\partial \Omega$.

Let $z\in\Omega$ and $0\ne\xi\in\partial U(z)$. Choose $\tau > 0$ such that $y = z+\tau\xi/|\xi| \in \partial \Omega$. 
Since $\xi \in \partial U(z)$, continuity gives
\[
0 = U(y) \geq U(z) + \xi \cdot (y-z) = U(z) + \tau|\xi|.
\]
By \eqref{eq:distance-barrier},
\[
 \tau|\xi|\leq-U(z)\leq Ld(z)\leq L\tau.
\]
Thus, every $\xi \in \partial U(z)$ satisfies $|\xi| \leq L$. For $z, z' \in \Omega$, choose $\xi \in \partial U(z)$ and $\xi' \in \partial U(z')$. Then, 
\[
\xi \cdot (z' - z) \leq U(z') - U(z) \leq \xi' \cdot (z' - z).
\]
Hence $|U(z) - U(z')| \leq L|z - z'|$. By continuity, the same bound holds on $\overline{\Omega}$, proving the Lipschitz bound.

Now, since $U\leq w$, we have $U^*\geq w^*$. On the other hand, for every $\xi$, the affine function $z\mapsto\xi\cdot z-w^*(\xi)$ lies below $w$ and hence below $U$. Thus $U^*(\xi)\leq w^*(\xi)$. Finally, if $w \not\equiv 0$, then $U\leq w\leq0$ implies that $U \not\equiv 0$. The proof is complete.
\end{proof}

\begin{samepage}
\subsection{An integral inequality for Legendre transforms}

We next compare the Legendre transforms of two functions $-v$ and $-h$ with matching endpoint values. The comparison function $h$ is concave; $v$ need only be Lipschitz. Equality holds if $v$ is also concave. $v', h'$ denote almost-everywhere derivatives of $v, h$, respectively, and $s_+=\max\{s,0\}$.

\begin{lemma}\label{lem:one-dimensional}
Let $v,h\in C^{0,1}(J)$, where $J=[\ell,r]\subset\R$ and $\ell<r$. Assume that $h$ is concave and
\[
 v(\ell)=h(\ell),\qquad v(r)=h(r).
\]
For $q\in\R$, set
\[
 S_v(q)=\max_{x\in J}\{qx+v(x)\},
 \qquad S_h(q)=\max_{x\in J}\{qx+h(x)\}.
\]
Then $S_v-S_h$ has compact support and
\begin{equation}
 \int_\R(S_v-S_h)\dd q
 \leq\frac12\int_\ell^r(v'^2-h'^2)\dd x.
 \label{eq:one-dimensional}
\end{equation}
Equality holds if $v$ is concave.
\end{lemma}
\end{samepage}

\begin{proof}
Choose $R>\max\{\lip{v}{J},\lip{h}{J}\}$. For $q\geq R$, both maxima occur at $r$; for $q\leq-R$, both occur at $\ell$. Since $v = h$ at the endpoints, $S_v=S_h$ outside $[-R,R]$.

The absolute continuity of Lipschitz functions gives
\begin{align}
 S_v(q)
 &=q\ell+v(\ell)+\max_{x\in J}\int_\ell^x(q+v'(s))\dd s\notag\\
 &\leq q\ell+v(\ell)+\int_\ell^r(q+v'(s))_+\dd s.
 \label{eq:positive-part}
\end{align}
If $v$ is concave, let $c\in J$ maximize $x\mapsto qx+v(x)$.
Then $q+v'\geq0$ almost everywhere on $(\ell,c)$
and $q+v'\leq0$ almost everywhere on $(c,r)$.
Hence equality holds in \eqref{eq:positive-part}.

For $|\tau|\leq R$, we have $\int_{-R}^R(q+\tau)_+\dd q=(R+\tau)^2/2$. Integrating \eqref{eq:positive-part} in $q$ therefore yields
\begin{equation}
 \int_{-R}^R S_v(q)\dd q
 \leq\frac{R^2}{2}(r-\ell)+R\bigl(v(\ell)+v(r)\bigr)
       +\frac12\int_\ell^r v'^2\dd x,
 \label{eq:one-dimensional-truncated}
\end{equation}
with equality when $v$ is concave. The corresponding formula for $h$ is an equality. Subtracting it cancels the terms involving $R$ and proves the assertion.
\end{proof}

\subsection{The partial Legendre transform and the energy estimate}\label{sec:partial}

Let $P\subset\R^2$ be a convex polygon. Write its vertical sections as
\begin{equation}
 \overline{P}=\{(x,y):x\in J,\ a(x)\leq y\leq b(x)\},\qquad J=[\ell,r].
 \label{eq:fiber-domain}
\end{equation}
The functions $a,b$ are Lipschitz and piecewise affine, with $a$ convex, $b$ concave, and $a<b$ on $(\ell,r)$. An endpoint section may be a point or a vertical side.

For a Lipschitz function $w$ on $\overline{P}$, define its partial Legendre transform in $y$ and the support function of the section by
\begin{align}
 F(x,p)&=\max_{a(x)\leq y\leq b(x)}\{py-w(x,y)\},\label{eq:partial-transform}\\
 H(x,p)&=\max_{a(x)\leq y\leq b(x)}py
       =\begin{cases}pb(x),&p\geq0,\\pa(x),&p\leq0.\end{cases}
 \label{eq:fiber-support}
\end{align}
Write $z=(x,y)$ and $\xi=(q,p)$ for the primal and dual variables. Subscripts denote almost-everywhere partial derivatives. For each $p$, $H(\cdot,p)$ is concave. If $w$ is convex, then $F(\cdot,p)$ is also concave. 

We now prove the energy estimates for the convex envelopes.

\begin{samepage}
\begin{proposition}\label{prop:envelope-energy}
Let $P\subset\R^2$ be a convex polygon.
Suppose $w\in C^{0,1}(\overline{P})$, $w\leq0$ on $\overline{P}$, and $w=0$ on $\partial P$. Let 
$U$ be the convex envelope of $w$ in $\overline{P}$,
and let $F,H$ be given by \eqref{eq:partial-transform}--\eqref{eq:fiber-support}. Then
\begin{align}
 I[U;P]&\leq\frac32\int_\R\int_\ell^r(F_x^2-H_x^2)\dd x\dd p,
 \label{eq:envelope-energy}\\
 \int_P|U|^3\dd x\dd y&\geq\int_P|w|^3\dd x\dd y.
 \label{eq:envelope-mass}
\end{align}
If $w$ is convex, then $U=w$ and equality holds in \eqref{eq:envelope-energy}. If $w\not\equiv0$, then $U\in\cA(P)$.
\end{proposition}
\end{samepage}

\begin{proof}
By \cref{lem:convex-envelope}, $U$ is Lipschitz, convex, and zero on $\partial P$, with $U\leq w\leq0$ and $U^*=w^*$. This proves \eqref{eq:envelope-mass}. Also, $w \not\equiv 0$ implies $U \in \mathcal{A}(P)$. We prove the energy estimate in three steps.

\medskip
\emph{Step 1. Regularity and behavior for large $|p|$.}
For $y\in[a(x),b(x)]$, let $y'$ be its nearest point in $[a(x'),b(x')]$. Then
$|y-y'|\leq\max\{|a(x)-a(x')|,|b(x)-b(x')|\}$.
Compare $py-w(x,y)$ with $py'-w(x',y')$, take maxima, and interchange $x,x'$. The Lipschitz continuity of $a,b,w$ shows that, for every $R>0$, there is a constant $C_R$ such that
\begin{equation}
 |F(x,p)-F(x',p)|\leq C_R|x-x'|\qquad\text{if   }|p|\leq R.
 \label{eq:F-lipschitz}
\end{equation}
Since $P$ is bounded, the defining maxima give a uniform Lipschitz bound in $p$. Together with the preceding estimate in $x$, this proves continuity of $F$. Since $w\leq0$ and the endpoint sections lie on $\partial P$,
\begin{equation}
 F\geq H,\qquad F(\ell,p)=H(\ell,p),\qquad F(r,p)=H(r,p).
 \label{eq:F-endpoints}
\end{equation}
Set $L=\lip{w}{\overline{P}}$. For $p>L$, the function $y\mapsto py-w(x,y)$ is strictly increasing; for $p<-L$, it is strictly decreasing. Its maximum is therefore at the corresponding boundary point, where $w=0$. Thus
\begin{equation}
 F(x,p)=H(x,p)\qquad\text{if   }|p|>L.
 \label{eq:F-tails}
\end{equation}
Choose $R>L$. The difference $F_x^2-H_x^2$ vanishes outside $J\times[-R,R]$, and both derivatives are bounded on this strip. This proves the absolute integrability of the difference.

\medskip
\emph{Step 2. The energy upper bound.}
Successive maximizations give
\[
 w^*(q,p)=\max_{x\in J}\{qx+F(x,p)\},\qquad
 h_P(q,p)=\max_{x\in J}\{qx+H(x,p)\}.
\]
For each $p$, apply \cref{lem:one-dimensional} using \eqref{eq:F-endpoints} and concavity of $H(\cdot,p)$. We obtain
\begin{equation}
 \int_\R\bigl(w^*(q,p)-h_P(q,p)\bigr)\dd q
 \leq\frac12\int_\ell^r(F_x^2-H_x^2)\dd x.
 \label{eq:dual-slice}
\end{equation}
By \cref{lem:full-dual}, $w^*-h_P=U^*-h_P$ is nonnegative and compactly supported. Integrating \eqref{eq:dual-slice} in $p$ and using \eqref{eq:full-dual} with $n=2$ proves \eqref{eq:envelope-energy}.

\medskip
\emph{Step 3. Equality for convex functions.}
If $w$ is convex, then $U=w$. Extend $w$ by $+\infty$ outside $\overline{P}$. For fixed $p$, the function $w(x,y)-py$ is convex, so
\[
 -F(x,p)=\inf_y\{w(x,y)-py\}
\]
is convex in $x$. Thus $F(\cdot,p)$ is concave, and equality holds in \eqref{eq:dual-slice} by \cref{lem:one-dimensional}. Integrating gives equality in \eqref{eq:envelope-energy}.
\end{proof}

\section{Proof of Theorem \ref{thm:main}}\label{sec:chords}
We first use the energy estimate in \cref{prop:envelope-energy} to establish a neighbor-decreasing property of the Monge--Amp\`ere eigenavlue for a class of area-preserving deformations of convex polygons. This property is the key for the reduction from polygons to triangles. We prove Theorem \ref{thm:main} at the end of the section. 

\subsection{Comparison with the endpoints}

Let $P_0$ be a convex polygon as in \eqref{eq:fiber-domain}, and let $\eta:J\to\R$ be Lipschitz and piecewise affine. Fix $t_-<0<t_+$. For $t\in[t_-,t_+]$, set
\begin{equation}
 P_t=\{(x,y+t\eta(x)):(x,y)\in P_0\}.
 \label{eq:movement}
\end{equation}
Assume that every $P_t$ is convex.  The boundary functions are $a_t=a+t\eta$ and $b_t=b+t\eta$, so the chord lengths are unchanged and
\begin{equation}
 |P_t|=\int_\ell^r(b_t-a_t)\dd x=|P_0|.
 \label{eq:area-preserved}
\end{equation}
Extending $\eta$ constantly outside $J$ makes the map in \eqref{eq:movement} a bi-Lipschitz homeomorphism of $\R^2$. It maps the closure and boundary of $P_0$ onto those of $P_t$.

We now compare $\lambda[P_0]$ with the Monge--Amp\`ere eigenvalues at the two endpoints by using the energy estimates from Proposition~\ref{prop:envelope-energy}.
\begin{theorem}\label{thm:endpoint-comparison}
For every family \eqref{eq:movement} satisfying these assumptions, at least one endpoint has no larger Monge--Amp\`ere eigenvalue than $P_0$:
\begin{equation}
 \min\{\lambda[P_{t_-}],\lambda[P_{t_+}]\}\leq\lambda[P_0].
 \label{eq:endpoint-comparison}
\end{equation}
\end{theorem}

\begin{proof}
Let $u$ be a nonzero convex eigenfunction on $P_0$. It is Lipschitz by \cite[Theorem 1.1(i)]{Le2026}. Put
\[
 E=I[u;P_0],\qquad M=\int_{P_0}|u|^3\dd x\dd y>0,
 \qquad E=\lambda[P_0]M.
\]
We will find an affine upper bound for $\lambda[P_t]$ that agrees with it at $t=0$. The sign of its slope determines which endpoint to choose. Let $V$ be the partial Legendre transform of $u$, and let $H$ be the support function of the sections of $P_0$.

\medskip
\emph{Step 1. A convex test function on $P_t$.}
We transport $u$ and take its convex envelope. Define
\begin{equation}
 w_t(x,y)=u(x,y-t\eta(x)),\qquad
 U_t= \mbox{the convex envelope of $w_t$ in $\overline P_t$}.
 \label{eq:transport}
\end{equation}
Then $w_t$ is Lipschitz, negative in $P_t$, and zero on $\partial P_t$. Changing variables in each section gives 
\begin{equation}
 \int_{P_t}|w_t|^3\dd x\dd y=M.
 \label{eq:transport-mass}
\end{equation}
By \cref{prop:envelope-energy}, $U_t\in\cA(P_t)$ and its $L^3$ integral is at least $M$. If $F_t$ denotes the partial Legendre transform of $w_t$, then
\begin{equation}
 F_t=V+tp\eta(x),\qquad H_t=H+tp\eta(x).
 \label{eq:transport-dual}
\end{equation}
Thus, for almost every $(x,p)$,
\begin{equation}
 (F_t)_x^2-(H_t)_x^2
 =V_x^2-H_x^2
 +2tp\eta'(x)(V_x-H_x).
 \label{eq:quadratic-cancellation}
\end{equation}
 Define
\begin{equation}
 B=3\int_\R\int_\ell^r
 p\eta'(x)(V_x-H_x)\dd x\dd p.
 \label{eq:linear-coefficient}
\end{equation}
The integrand vanishes for $|p|>\lip{u}{\overline{P_0}}$ and is bounded on $[\ell, r]\times [-\lip{u}{\overline{P_0}}, \lip{u}{\overline{P_0}}]$ by Step~1 of the proof of \cref{prop:envelope-energy}. Hence $B$ is finite. 

Integrating \eqref{eq:quadratic-cancellation}, and using the equality case of \cref{prop:envelope-energy} for $u$, gives
\begin{equation}
 \frac32\int_\R\int_\ell^r\bigl((F_t)_x^2-(H_t)_x^2\bigr)\dd x\dd p
 =E+tB.
 \label{eq:affine-energy}
\end{equation}
Applying the inequality case to $w_t$, we obtain
\begin{equation}
 I[U_t;P_t]\leq E+tB,\qquad
 \int_{P_t}|U_t|^3\dd x\dd y\geq M.
 \label{eq:competitor-bounds}
\end{equation}

\medskip
\emph{Step 2. Choosing an endpoint.}
Since $\lambda[P_t]>0$, the variational characterization of the Monge--Amp\`ere eigenvalue in  \eqref{eq:rayleigh} gives
\[
 \lambda[P_t]M
 \leq\lambda[P_t]\int_{P_t}|U_t|^3\dd x\dd y
 \leq I[U_t;P_t]\leq E+tB.
\]
Dividing by $M$ yields
\begin{equation}
 \lambda[P_t]\leq\lambda[P_0]+t\frac{B}{M}.
 \label{eq:upper-support}
\end{equation}
If $B\geq0$, apply \eqref{eq:upper-support} at $t=t_-<0$ to obtain $\lambda[P_{t_-}]\leq\lambda[P_0]$. If $B<0$, use $t=t_+>0$ to obtain $\lambda[P_{t_+}]\leq\lambda[P_0]$. If $B=0$, either endpoint may be chosen. This proves \eqref{eq:endpoint-comparison}.
\end{proof}

\begin{remark}
From \eqref{eq:upper-support}, we can show that $\lambda[P_t]$ is concave in $t\in [t_{-}, t_+]$. Indeed, let $t_0\leq t_1\in [t_{-}, t+]$ and $\theta\in (0, 1)$. By shifting the arguments in $t$, we might assume that $(1-\theta) t_0 + \theta t_1=0$. Now, apply \eqref{eq:upper-support} at $t_0$ and $t_1$, multiply by $(1-\theta)$ and $\theta$, and add. We obtain
\[(1-\theta) \lambda[P_{t_0}] + \theta \lambda [P_{t_1}] \leq \lambda[P_0]=\lambda [P_{(1-\theta) t_0 + \theta t_1}].\]
\end{remark}
\subsection{Reduction to triangles}\label{sec:polygons}
The following lemma allows us to replace a convex polygon with at least four vertices by one with fewer vertices that has the same area and no larger Monge--Amp\`ere eigenvalue.

\begin{lemma}\label{lem:vertex-removal}
For every convex polygon $P$ with $m\geq4$ vertices, there is a convex polygon $P'$ with at most $m-1$ vertices such that
\begin{equation}
 |P'|=|P|,\qquad \lambda[P']\leq\lambda[P].
 \label{eq:one-step}
\end{equation}
\end{lemma}

\begin{proof}
Choose a vertex $B$ of $P$, with neighbors $A$ and $C$. Move $B$ parallel to the segment $\overline{AC}$ and keep all other vertices fixed. Continue in both directions as long as the polygon remains convex. Notice that such movement preserves area and that when $B$ reaches an endpoint, the polygon has $m-1$ vertices. By Theorem \ref{thm:endpoint-comparison}, one of the endpoint polygon $P'$ satisfies $\lambda[P'] \leq \lambda[P]$. To apply Theorem \ref{thm:endpoint-comparison}, we need to describe the map $\eta$ in  \eqref{eq:movement}. Choose coordinates so that
\[
A=(a,\alpha),\qquad C=(a,\gamma),\qquad B=(b,\beta),
\qquad a<b,
\]
and all vertices other than $B$ lie in $\{x\le a\}$.
It suffices to choose
\[
\eta(x)=
\begin{cases}
0, & x\le a,\\[2pt]
\dfrac{x-a}{b-a}, & a<x<b,\\[6pt]
1, & x\ge b.
\end{cases}
\]

The proof is complete.
\end{proof}

Finally, we are ready to prove \cref{thm:main}.

\begin{proof}[Proof of \cref{thm:main}]
By \cref{lem:affine}, we can assume $|\Omega| = 1$. Fix a triangle $T$ of area one.

\medskip
\emph{Case 1. $\Omega$ is a convex polygon.}
Applying \cref{lem:vertex-removal} iteratively
gives a triangle $\Delta $ with $|\Delta|=|\Omega|$ and
$\lambda[\Delta]\leq\lambda[\Omega]$.
By \cref{lem:affine}, we have
\begin{equation}
 \lambda[\Omega] = \lambda[\Omega]|\Omega|^2
 \geq\lambda[\Delta]|\Delta|^2
 =\lambda[T].
 \label{eq:polygon-minimum}
\end{equation}

\medskip
\emph{Case 2. $\Omega$ is a general convex domain.}
By \cite[Theorem~1.8.16]{Sch2014}, we can choose convex
polygons $\{\Omega_k\}_{k=1}^\infty$ such that
$\overline{\Omega_k}$ converges to $\overline{\Omega}$
in the Hausdorff distance and $|\Omega_k|\to|\Omega|$.
By \cite[Proposition~5.8(i)]{Le2018}, we also have
$\lambda[\Omega_k]\to\lambda[\Omega]$. Using \cref{lem:affine}, we then apply Case 1 and pass to the limit, and obtain
\begin{equation}
 \lambda[\Omega] = \lambda[\Omega]|\Omega|^2
 =\lim_{k\to\infty}\lambda[\Omega_k]|\Omega_k|^2
 \geq\lambda[T].
 \label{eq:invariant-main}
\end{equation}
The proof is complete.
\end{proof}

\medskip
\begin{remark}
Our argument also proves the planar centrally symmetric case of Le's conjecture \cite[Conjecture~1.5(ii)]{Le2018}. We call a set $K\subset\R^n$ centrally symmetric if $-x\in K$ whenever $x\in K$.
In \cref{lem:vertex-removal}, take a centrally symmetric polygon with at least six vertices.
Move two opposite vertices at equal speeds in opposite directions, each parallel to its nearest diagonal.
This preserves area and central symmetry. (In more details, let
$P=-P$ be a convex polygon with at least six vertices, and let
$A,B,C$ be consecutive vertices in counterclockwise order. We need to describe the map $\eta$ in  \eqref{eq:movement}. Let $Q$ be the convex hull of the vertices of $P$ other than $B$ and $-B$.
The remaining vertices contain at least two nonparallel antipodal
pairs. Thus $Q$ has nonempty interior, is centrally symmetric, and
contains the origin in its interior. After rotating coordinates, we may write
\[
 A=(a,\alpha),\qquad C=(a,\gamma),\qquad B=(b,\beta),
 \qquad 0<a<b,\quad \alpha<\gamma.
\]
Then $Q\subset\{|x|\le a\}$, and the projection of $\overline P$
onto the $x$-axis is $[-b,b]$. Put $h=b-a$ and define
\[
 \eta(x)=
 \begin{cases}
  (x+a)/h, & -b\le x\le-a,\\
  0,       & -a\le x\le a,\\
  (x-a)/h, & a\le x\le b.
 \end{cases}
\]
Extend $\eta$ constantly beyond $[-b,b]$. This function $\eta$ is odd,
Lipschitz, and piecewise affine, and we can apply Theorem \ref{thm:endpoint-comparison}.)
By \cref{thm:endpoint-comparison}, one endpoint polygon has no larger Monge--Amp\`ere eigenvalue and at least two fewer vertices.
By iteration, the problem reduces to a parallelogram, which has the same Monge--Amp\`ere eigenvalue as a square of the same area by the affine invariance.
Then, an approximation by centrally symmetric polygons gives the following theorem.
\end{remark}

\begin{theorem}[Square minimizes the Monge--Amp\`ere eigenvalue among centrally symmetric bounded planar convex domains of a given area]
Let $\Omega\subset\R^2$ be a centrally symmetric bounded convex domain with nonempty interior, and let $S\subset\R^2$ be a square with area $|S|=|\Omega|$. Then, their Monge--Amp\`ere eigenvalues satisfy
\[
 \lambda[\Omega]\geq\lambda[S].
 \label{eq:main2}
\]
\end{theorem}

\end{document}